\documentclass[1p,authoryear]{elsarticle}
\usepackage{amsmath,natbib}
\usepackage{fullpage}
\usepackage{amsfonts}
\usepackage{amssymb}
\usepackage{amsthm}

\numberwithin{equation}{section}
\theoremstyle{plain}
\newtheorem{thm}{Theorem}[section]
\newtheorem{lemma}{Lemma}[section]

\newtheorem{definition}{Definition}[section]
\newdefinition{rmk}{Remark}
\newproof{pf}{Proof}
\begin{document}

\title{The Fourier dimension of Brownian limsup fractals}
\author{Paul Potgieter}
\ead{potgip@unisa.ac.za}
\address{Department of Decision Sciences, University
of South Africa, P.O. Box 392, Pretoria 0003, South Africa}
\cortext[cor1]{Tel: +27 12 433 4603}

\begin{abstract}
Robert Kaufman's proof that the set of rapid points of Brownian motion has a Fourier dimension equal to its Hausdorff dimension was first published in 1974. A study of the proof of the original paper revealed several gaps in the arguments and a slight inaccuracy in the main theorem. This paper presents a new version of the construction and incorporates some recent results in order to establish a corrected version of Kaufman's theorem. The method of proof can then be extended to show that functionally determined rapid points of Brownian motion also form Salem sets for absolutely continuous functions of finite energy.
\end{abstract}

\maketitle 

\section{Introduction} 
\label{intro}

Brownian motion has proved itself a rich source of sets with interesting dimensional properties. It is well known that the level sets have Hausdorff dimension $1/2$, and it has recently been shown that they have equal Fourier dimension in \citet{FoucheMukeru}, implying that they form Salem sets. It has been accepted since 1974 that the rapid points of Hausdorff dimension also form Salem sets, due to the work of \citet{OreyTaylor} and \citet{Kaufman}. A study of Kaufman's paper on the Fourier dimension of the set of rapid points has, however, raised several questions regarding the proof. Due to the significance of the result to the study of sample path properties of Brownian motion, it was felt that an updated and expanded version of the proof was merited. Furthermore, such methods can shed light on random fractals whose Hausdorff dimensions have been found, but whose Fourier dimensions are not yet known. 

As in \citet{Potgieter}, we take the following as the definition of one dimensional Brownian motion:

\begin{definition}\label{def:1.1}
Given a probability space $(\Omega, \mathcal{B}, \mathbf{P})$, a
{\emph{Brownian motion}} is a stochastic process $X$ from
$\Omega \times [0,1]$ to $\mathbb{R}$ satisfying the following
properties:
\begin{enumerate}
\item{Each path $X(\omega, \cdot): [0,1]\to \mathbb{R}$ is almost
surely continuous} 
\item{$X(\omega,0) = 0$ almost surely}
\item{For $0\leq t_1 <t_2 \cdots < t_n \leq 1$, the random
variables $X(\omega, t_1 ), X(\omega ,t_2)- X(\omega , t_1),\dots,
X(\omega , t_n )- X(\omega ,t_{n-1})$ are independent and normally
distributed with mean $0$ and variances $t_1, t_2 - t_1 ,\dots ,t_n -
t_{n-1}$, respectively.}
\end{enumerate}
\end{definition}

We let $X$ denote a continuous Brownian motion on $[0,1]$, and let $X(t)$ denote the value of a sample path
$X(\omega)$ at $t$, if there is no confusion as to which sample path
is used. We define the random set of rapid points relative to the
parameter $\alpha$ (also referred to as the $\alpha$-rapid points) by
\begin{equation}\label{eq:1.1}E_{\alpha} (\omega) = \left\{t\in [0,1]: \limsup_{h\to 0}
\frac{|X(t+h)-X(t)|}{\sqrt{2|h|\log{1/|h|}}} \geq
\alpha\right\}.\end{equation}

Again, we discard $\omega$ and just write $E_{\alpha}$ when there is
no confusion as to the sample path. The sets $E_{\alpha}$ have
Lebesgue measure $0$ almost surely. They are exceptional points of rapid growth, since the
usual local growth behaviour is described by Khintchine's law of the
iterated logarithm 3\citep{Khintchine}:
\begin{equation}\label{eq:1.2}\mathbf{P}\left\{\limsup_{h\to 0} \frac{X(t_0
+h)-X(t_0)}{\sqrt{2|h|\log{\log{1/|h|}}}}=1\right\}=1,\end{equation} for any
prescribed $t_0$.

\citet{OreyTaylor} showed that the sets
$E_{\alpha}$ have Hausdorff dimension $1-\alpha^2$. An elementary proof of the result can be found in~\cite{Potgieter}.

In this paper we define the Fourier-Stieltjes transform of a measure $\mu$ with support $A\subseteq \mathbb{R}$ as \begin{equation}
\hat{\mu}(u) = \int_{A}e^{ixu}d\mu (x).
\end{equation}

We have the
following definition:
\begin{definition}\label{def:1.2}
An $M_{\beta}$-set is a compact set in $\mathbb{R}^n$ which carries
a measure $\mu$ such that $\hat{\mu}(u) =\textrm{o}(|u
|^{-\beta})$ as $| u | \to \infty$, for some $\beta >0$. For a compact set $E$, we call
the supremum of the $\alpha$ such that $E$ is an $M_{\alpha
/2}$-set, the \emph{Fourier dimension} of $E$. We shall denote this
by $\textrm{dim}_F E$.
\end{definition}

The Fourier dimension is a somewhat more elusive character than
Hausdorff dimension. They are usually different, as for instance in
the case of the triadic Cantor set, which has strictly positive Hausdorff
dimension but a Fourier dimension of $0$ \citep{Kechris}. Indeed, it
can be shown that Hausdorff dimension majorises Fourier dimension.
When the dimensions coincide, the set is called a Salem set. Kaufman's paper claims to establish the following result:

\begin{thm}\label{thm:2.1}
\citep{Kaufman} With probability 1, a certain compact
subset of $E_{\alpha}$, the $\alpha$-rapid points of a given
Brownian motion $X$, carries a probability measure $\mu$ such that
$\hat{\mu} (\xi ) = o(| \xi |^{\frac{1}{2}(\alpha^2 -1)}).$ 
\end{thm}

It should be noted that this result claims that the supremum referred to in the definition of the Fourier dimension is actually attained, which does not seem to obviously follow from the arguments in the original paper. Instead, we prove a slightly weakened version of Kaufman's theorem, although the result still implies the equality of the Hausdorff and Fourier dimensions. The main result of this paper is the following:

\begin{thm}\label{thm:1.1} Given $\alpha$, $0<\alpha<1$, then for each $\gamma < 1 - \alpha^2$ there exists, with probability $1$, a measure $\mu$ on a compact subset of $E_{\alpha}$ such that
$\hat{\mu} (u) = o(|u|^{-\frac{\gamma}{2}})$ That is, $E_{\alpha}$ is a set of Fourier dimension
$1-\alpha^2$.
\end{thm}

(The fact that Hausdorff dimension majorises Fourier dimension guarantees that the Fourier dimension of $E_{\alpha}$ cannot be larger than $1-\alpha^2$.)

The structure of the proof is initially similar to Kaufman's. Specifically, we use two of the same lemmas that Kaufman did, albeit slightly altered.
These do not require innovative techniques to prove, only a refinement of Kaufman's original arguments. We prove the first of these in Section 2. The second lemma, the domain of which had to be slightly limited, is proved in Section 3. This is followed by the proof of the main theorem, where  we need to depart from Kaufman's construction to a larger extent. The measures constructed in the various stages in his paper do not necessarily converge to a probability measure, and it is only inferred that the limit measure has the requisite properties, never proven. It has proved difficult to mend the construction using only the Fourier-analytical methods originally considered. It became necessary to introduce certain dimensional arguments \citep{Potgieter} which give us estimates on the actual number of intervals considered at each stage. In this way, it is quite obvious that the Hausdorff dimension of the set is instrumental in guaranteeing the desired Fourier dimension, and in such a precise way that the set becomes Salem. In Section 4, we define functionally determined rapid points of Brownian motion, which have properties similar to those of the sets $E_{\alpha}$. The Hausdorff dimension of such sets was found by \citet{DeheuvelsMason}, in an extension of the Orey-Taylor theorem. Further interesting extensions were obtained \citet{KhosPerX}. We extend Kaufman's result to show that the functionally determined rapid points also form Salem sets, almost surely.

\section{Sums of random variables}

In this section, we establish a lemma necessary for the proof of Theorem \ref{thm:1.1}, in addition to some supporting inequalities. Suppose that
$\{\xi_n : 1\leq n \leq m\}$ is a finite set of independent random variables with
a common distribution
\[\mathbf{P}\{\xi_n =1\}=p=1-\mathbf{P}\{\xi_n =0\}.\] We want
estimates for sums of the form $\sum (p-\xi_n )a_n$, $a_n \in \mathbb{C}$. Let $\sigma^2
=\sum_{n=1}^{m} |a_n |^2$ and $B=\max_{n=1,\dots ,m}|a_n |$.
\begin{lemma}\citep{Kaufman}\label{lem:2.1}
Provided that $YB<2p\sigma^2$ and all else as above,
\begin{equation}\label{eq:2.2} \mathbf{P}\left\{ |\sum_{n=1}^{m} (p-\xi_n )a_n|\geq Y\right\} \leq
4e^{-\frac{1}{16} p^{-1}
\sigma^{-2}Y^2}.\end{equation}
\end{lemma}
\begin{pf} We use a basic inequality that can be easily
obtained by writing out the series expansion of $e$:
\begin{equation}\label{eq:2.2.1} pe^{t(1-p)}+(1-p)e^{-pt} \leq 1+p(1-p)t^2 \leq e^{t^2 p},\end{equation}
which is valid for $0\leq p \leq 1$, $-1\leq t \leq 1$. We now turn
to Chebyshev's inequality, in the form
\begin{equation} \mathbf{P}\{X\geq k\}\leq \frac{1}{k}\mathbf{E}(X)\end{equation}
for any random variable $X$ with finite expectation.  Kaufman
mentions that he deduces the final inequality \ref{eq:2.2} by using Chebyshev; we feel
it is instructive to elaborate on this, in the interest of completeness. We now give the
details required to obtain an estimate of the probability
\begin{equation}
\mathbf{P}\left\{ |\sum_{n=1}^{m}(p-\xi_n) a_n|\geq Y\right\} .\end{equation}
(Where the range of the summation and product is clear, we will omit it from now on.) Note that this will be the same as the probability
\begin{equation} \mathbf{P}\{e^{t|\sum(p-\xi_n) a_n|}\geq e^{tY}\} \textrm{ for any } 0< t \leq 1.\end{equation}
Set $Z=\sum(p-\xi_n )a_n$, where the $a_n$ are, for the moment, assumed to be real-valued. Because of the assumed independence of
the $\xi_n$, the expected value of $e^{tZ}$ can be evaluated by the
integral 
\begin{eqnarray}
\int e^{tZ}d\mathbf{P} &=& \prod_{n=1}^{m} \int e^{(p-\xi_n)a_n
t}d\mathbf{P}\nonumber \\
&=& \prod (pe^{(1-p)a_n t}+(1-p)e^{-pa_n t}) \nonumber \\
&\leq & \prod e^{pt^2 a_{n}^{2}} \nonumber \\
&=& e^{pt^2\sigma^2}\label{eq:2.2.0}
\end{eqnarray}
(the range of integration is understood to be the the entire probability space over which the $\xi_n$ are defined).

Note that we have used the basic inequality \ref{eq:2.2.1} repeatedly with $a_n t$
instead of just $t$, and must require that $0\leq \max |a_n | t \leq
1$. It then follows from Chebyshev that
\begin{equation}  \mathbf{P}\{Z\geq Y\}=\mathbf{P}\{e^{tZ}\geq e^{tY}\}\leq e^{t^2 p\sigma^2 -tY}.\end{equation}
By symmetry in the use of $e^{tZ}$ and $e^{-tZ}$ in \ref{eq:2.2.0}, we can conclude that
\begin{equation} \mathbf{P}\{|Z|\geq Y\}\leq
2e^{pt^2\sigma^2-tY}.\end{equation} We now choose a specific
value of $t$ to minimise the right hand side of the inequality. It
is easily seen that $t=Y/2p\sigma^2 $ is such a value (and still satisfies
the condition $0\leq Bt\leq 1)$. By substituting
for $t$, this yields
\begin{equation}\label{eq:2.2.2} \mathbf{P}\{|Z|\geq Y\} \leq 2e^{ -\frac{1}{4}p^{-1}\sigma^{-2}Y^2} ,\end{equation} as long as $YB\leq 2p\sigma^2 $. We do
still require an inequality for complex values of $a_n$ and can
obtain a rough but useful one by considering separate probabilities for the real and imaginary parts of the terms $a_n$. This gives us a probability

\begin{equation}
\mathbf{P}\{|Z|\geq Y\} \leq 4e^{ -\frac{1}{16}p^{-1} \sigma^{-2}Y^2}
\end{equation} 
which proves the lemma. 
\end{pf}

To proceed with the proof of the
theorem we need some further standard inequalities. Let
\begin{equation} S(a,b)=\max |X(b)-X(a)|, \quad 0\leq a < b\leq 1.\end{equation}
The exact distribution of $S$ was found by \citet{Feller}, but
we need only an approximation of the following form:
\begin{equation}\label{eq:2.4}\mathbf{P}\{S\geq Y\}=e^{-\frac{1}{2}Y^2}e^{o(Y^2)}, \quad
Y\to \infty.\end{equation} 
This can be established by using the stationarity of intervals, the reflection principle (e.g. on p26 of \citet{ItoMcKean}) and, for instance, Lemma 4 on p11 of \citet{Freedman}. Since this approximation will be used later, we state it here:
\begin{equation}\label{eq:2.5}\left( \frac{1}{Y} -\frac{1}{Y^3}\right) e^{-\frac{1}{2}Y^2} \leq \frac{1}{\sqrt{2\pi}}\int_{Y}^{\infty}e^{-\frac{1}{2}x^2}dx \leq
 \frac{1}{Y} e^{-\frac{1}{2}Y^2}, \quad Y>0.\end{equation}

We will use inequality \ref{eq:2.4} in the following form:
\begin{equation} \label{eq:2.7} \mathbf{P}\left\{\max_{0\leq s\leq h} X(s)\geq \lambda h^{\frac{1}{2}}\right\} = h^{\beta^2 (1+o(1))}\end{equation} as $\lambda \to \infty$ (for $\lambda = \beta \sqrt{2\log{h^{-1}}}$).

\section{Large increments of Brownian motion}

We can now start our construction, although we first recall the initial stages of Kaufman's construction, the better to highlight differences between the two.  In the original proof, the construction of the first measure 
proceeded as follows: Let $0\leq r<s\leq 1$, $\beta <\alpha$ and let
$I_n$ be a division of the interval $(r,s)$ into $N$ equal intervals
of length $(s-r)N^{-1}$. Each interval $I_n$ is further subdivided
into intervals $I^{q}_{n}$ of length $(s-r)bN^{-1}$, where it is assumed
that $b^{-1}\in \mathbb{Z}$ and $1\leq q\leq b^{-1}$. An interval
$I^{q}_{n}$ with lower extremity $x$ is selected if
\begin{equation}\label{eq:2.9} X(x+h)-X(t) \geq (\beta -2b^{\frac{1}{2}})(2h\log h^{-1})^{1/2} \textrm{ on }
x\leq t \leq x+bh,\end{equation} where $h$ is the length of the interval. The
selection of the intervals $I^{q}_{n}$ ($1\leq n \leq N$) are
mutually independent for each $q$, with the probability $p=p_N \geq
N^{-\beta^2}N^{o(1)}$ for large $N$, which follows from \ref{eq:2.7}. Let $\mu_0$ be
Lebesgue measure on $[r,s]$ and let $\xi$ be the characteristic
function of the selected intervals. For a Borel subset $A$ of
$[r,s]$, define
\begin{equation}\label{def:2.1}
\mu_1 (A) = p^{-1} \xi(A) \mu_0 (A).
\end{equation}
(We suppose that $\xi(A)=1$ if there is some $x\in A$ such that $\xi
(x)=1$, and $\xi (A)=0$ otherwise.)

We will retain the above definition of $\mu_1$, however the definition of $\xi$ will need to change in order to continue the construction beyond the first step. 

Instead of considering a large fluctuation between the initial segment and final point 
of an interval as above, we now consider an interval to be ``rapid" at each stage of the construction 
if there exist any two points within it between which the fluctuation is large enough. That is,
we let $A$ be the subset of intervals of the form $[(k-1)N^{-1},kN^{-1}]$ of an equal division of $[0,1]$ into $N$
intervals for which
\begin{equation}\label{eq:2.11}
\exists s,t \in [(k-1)N^{-1},kN^{-1}] \left( N^{-\frac{1}{2}}|X(t)-X(s)|\geq \beta \sqrt{2\log N}\right),
\end{equation}
where $\beta < \alpha$.
Note that we no longer need the constant $b$ in this formulation. We let $\xi^{1}_n (u)=1$ if $u\in [(n-1)N^{-1},nN^{-1}]$ and the interval is rapid, and $0$ otherwise. For simplicity, we just write $\xi_n$. The new probability corresponding to this choice of interval clearly has the same lower bound (\ref{eq:2.7}) as the one considered above, and we will continue to use this lower bound throughout the rest of the paper. 

The following lemma differs in its statement from the version proposed by Kaufman principally by the range of the statement, i.e. the condition $u\geq 0$ has had to be replaced by $u> 1$. 

\begin{lemma}\label{lem:2.2}
Given $\varepsilon >0$, for large enough $N$ the inequality
\begin{equation}\label{eq:2.10}|\hat{\mu}_1 (u) -\hat{\mu}_0(u)|<\varepsilon
|u|^{\frac{1}{2}(\alpha^2 -1)}\end{equation} holds for all $|u|>1$, with
probability approaching $1$ as $N \to \infty$.
\end{lemma}
\begin{pf} We only prove the lemma for positive values of $u$; it will be clear from the construction that the result is symmetrical around the origin. Letting $f_n (u)$ denote the Fourier transform of Lebesgue measure, we get that
\begin{equation}\hat{\mu}_0 (u)-\hat{\mu}_1 (u) =\sum_{n=1}^{N}(1-p^{-1} \xi_n
(u))f_n (u),\end{equation} where $|f_n (u)| \leq 2|u|^{-1}$. (The index $n$ in $f_n (u)$ is clearly superfluous and we use it only to conform more closely to the notation of Lemma 2.1.)  We set
$C(u)=\max|f_n (u)|$ and rewrite the desired inequality as in Lemma 2.1:
\begin{equation}\label{eq:2.12}\left|\sum_{n=1}^{N} (p-\xi_n )f_n (u)\right|<\varepsilon
pu^{\frac{1}{2}(\alpha^2 -1)},\end{equation} where $B=C(u)$ and
$\sigma^{2} =NC^2 (u)$. We now divide up the positive reals greater than $1$ and prove that inequality \ref{eq:2.10} holds with probability  close to $1$ for each section. We do this by showing that the opposite inequality
\begin{equation}\label{eq:2.16}
\left|\sum (p-\xi_n )2u^{-1}\right|\geq \varepsilon pu^{\frac{1}{2}(\alpha^2 -1)}
\end{equation} occurs with a probability approaching $0$. Because $f_n (u)$ is dependent only on $u$ (and not $n$) and $|u|^{-1}<|f_n (u)| \leq 2|u|^{-1}$, this will imply that the reverse of \ref{eq:2.12} also tends to $0$ in probability. 

For $u\geq N$, $u^{\alpha^2 +1}<N^2$, we set $Y=\varepsilon p u^{\frac{1}{2}(\alpha^2 -1)}$.
To apply Lemma 2.1, we first need to verify that $YB<2p\sigma^2$. This is easily done with
$\varepsilon <1/4$ and $u^{\alpha^2 +1}<N^2$. In computing the upper bound to the exceptional probability given in \ref{eq:2.2}, we estimate the exponent:
\begin{eqnarray}
-\frac{1}{16}p^{-1}\sigma^{-2}Y^2 &=& -\frac{1}{256}N^{-1} \varepsilon^{2}p u^{\alpha^2 +1}\nonumber \\
&=&-cpN^{-1}u^{\alpha^2+1}\nonumber \\
&<& -cN^{o(1)}N^{-\beta^2-1}N^{\alpha^2+1}\nonumber \\
&=& -cN^{\alpha^2 -\beta^2 +o(1)}<-cN^{\delta},
\end{eqnarray}
for some $\delta>0$ (since $\beta <\alpha$). This implies that the probability of \ref{eq:2.16}  tends to $0$. 
If $1\leq u <(2\varepsilon^{-1}N)^{2/(1+\alpha^2)}$, we can use the lower bound on $|f_n (u)|$ to verify that $YB<2p\sigma^2$. The exceptional probability will be the same as the one calculated for the previous case. This is sufficient for the interval $1<  u<N^{2/(\alpha^2 +1)}$.

We now suppose that $u^{\alpha^2 +1}\geq N^2$ and $u\leq N^2$. In the proof of Lemma 2.1, it is shown that
\begin{equation} \mathbf{P}[|X|\geq Y] \leq 2e^{pt^2 \sigma^2 -tY},\end{equation}
where $t$ may be chosen according to our purposes. Choosing $t=\eta B^{-1}$ for some small
$\eta >0$ and with $B$, $\sigma^2$ and $Y$ as before, the exponent in the above becomes
\begin{eqnarray}
p\eta^2 B^{-2}\sigma^2-\eta B^{-1}Y&=& p\eta^2 N-\frac{1}{2}\eta \varepsilon p u^{\frac{1}{2}(\alpha^2 +1)} \nonumber \\
&\leq& \eta^2 p N-\frac{1}{2}\eta \varepsilon p N\nonumber \\
&\leq& N(\eta^2 -\frac{1}{2}\varepsilon \eta).
\end{eqnarray}

Since we can choose $\eta$ small enough so that $\eta^2 -\frac{1}{2}\varepsilon \eta<0$, the
probability once again behaves in the desired way.

We have therefore established that for each $u$, $1< u \leq N^2$, there exists some $\delta = \delta (u)$ such that the probability of \ref{eq:2.10} being violated is smaller than $ce^{-\delta N}$. Of course, it is possible that the exceptional sets differ for each $u$, and that the union over all the values of $u$ has an unacceptably large probability. To avoid this, we consider only the finite number of $u$ which can be represented as
$u=jN^{-2}$, $N^2 < j \leq N^4$. 

In the worst case, the total exceptional probability will be the sum of the exceptional probabilities for each $u$, yielding a probability of at most $O(N^2)e^{-\delta_1 N}$, where
$\delta_1 = \min \delta (u)$ for $u=  1+1/N, 1+2/N, \dots ,N^2$. This can clearly be made as small as necessary by increasing $N$. 

As long as the maximum fluctuation of the sum over the intervals between these points remains small, we will be able to conclude that the inequality \ref{eq:2.12} holds with large probability. In order to do so, we first examine the maximum possible fluctuation of $\sum_{n=1}^{N} (p-\xi_n )f_n (u)$ by considering the derivative:
\begin{equation}
\frac{d}{du} \sum_{n=1}^{N} (p-\xi_n )f_n (u) = \sum_{n=1}^{N} (p-\xi_n )\left( \frac{ie^{iu}}{u}-\frac{(e^{iu}-1)}{u^2} \right).
\end{equation}

We observe that the modulus of this expression is less than that of $\sum_{n=1}^{N} (p-\xi_n )f_n (u)$ (supposing that $u>1$). The maximum total fluctuation over an interval of length $N^{-2}$ is then on the order of $1/N^2$ times the function value. This can easily be accommodated in the theorem, for instance by considering $\varepsilon+\varepsilon'$ for some small $\varepsilon' < \varepsilon /2$ in the statement of the theorem, leaving the rest of the proof unchanged.

For $u> N^2$, the inequality \ref{eq:2.12} becomes simpler to establish, provided we use the estimate  from \ref{eq:2.25} for the probability $p$. Separating the terms on the left-hand side, we get
\begin{eqnarray}
|\hat{\mu}_1(u)|&< &2u^{-1}p^{-1}\sum \xi_n  \nonumber \\
&<& 2^{\frac{5}{2}}\beta (\log N)^{\frac{1}{2}}u^{-1}N^{1+\beta^2}\nonumber \\
&<& 2^{\frac{5}{2}}\beta (\log N)^{\frac{1}{2}}u^{-\frac{1}{2}(1-\beta^2)}
\end{eqnarray}
and
\begin{eqnarray}
|\hat{\mu}_0(u)|<2u^{-1}.
\end{eqnarray}
Since $\beta < \alpha$, for sufficiently large values of $N$, both of the above estimates are smaller than $\varepsilon u^{\frac{1}{2}(\alpha^2 -1)}/2$, implying that \ref{eq:2.10} is satisfied. 
\end{pf}

We now proceed to the proof of Theorem \ref{thm:1.1}. Assume that the
construction as described in the preamble to lemma \ref{lem:2.2} has been accomplished on the interval
$[0,1]$, with $N$ such that the lemma holds except on a set of
probability at most $\eta_1$. 

If we were to 
try to replicate this construction on a finer subdivision of the unit interval (by simply increasing $N$, say) as in the original proof, the measure thus constructed would not necessarily have a closed support contained 
in that of the first, and restricting the second measure in a way that would nest
the support may lead to problems regarding the total mass of the limit measure - that is, 
the measure may have a total mass of $0$. 

Instead, we let $\beta$ increase slightly to $\beta_2<\alpha$, and divide the interval $[0,1]$ into $N^2$ equal intervals. An interval of the form $[(k-1)N^{-2},kN^{-2}]$, $0<k\leq N^2$, now belongs to the set $A_2$ if

\begin{equation}\label{eq:2.13}
\exists s,t \in [(k-1)N^{-2},kN^{-2}] \left( N^{-\frac{1}{2}} |X(t)-X(s)| \geq \beta_2  \sqrt{2\log N} \right)
\end{equation}

If the probability of the above event is $p_2$, we construct the second measure as
$\mu_2 (A) = p_{2}^{-1} \xi^2 (A)\mu_0 (A)$, with $\xi^2$ being the
characteristic function applicable to the new division, as in the construction of $\mu_1$. The support of $\mu_2$ is contained in that of $\mu_1$, since it can be easily shown that if \ref{eq:2.13} holds for an interval of length $N^{-2}$, then \ref{eq:2.11} holds for the containing interval of length $N^{-1}$ 

We require that lemma \ref{lem:2.2} be true for the new partition and the new probability associated with the partition, all with a total exceptional probability of less than $\eta_{1}^{2}$. If we examine the proof of the lemma, all arguments still obtain if we substitute the original values for the new ones, since all the inequalities necessary to apply lemma \ref{lem:2.1} still hold. Also, the exceptional probability may be computed in the same way as in lemma \ref{lem:2.2}, but will now be at most $\eta_{1}^{2}$ because of the use of the larger value $N^2$. This remains true for each higher exponent of $N$, hence we may apply the lemma at each stage of the construction.

The rest of the construction proceeds similarly, using intervals of length $N^{-3}$, $N^{-4}$, and so on, such that the values $\beta_3$, $\beta_4 \dots$ tend to $\alpha$. By choosing $N$ sufficiently large we can ensure that the exceptional probability is small enough to guarantee a small total exceptional probability (smaller than some $\eta$ given at the start, for instance), where we calculate the total exceptional probability as the sum of all those at each stage.  

Next, we need to establish that the mass of the measure constructed at each step does not decrease to 0 or increase to $\infty$. We first show that the measure $ \mu_1 $ can be bounded from above:
\begin{eqnarray}
\| \mu_1 \|&=& \int_{0}^{1} d\mu_1 \nonumber \\
&= & | A |p_{1}^{-1} N^{-1} \nonumber \\
&\leq& | A |N^{\beta^2 -1} N^{o(1)},
\end{eqnarray}
where $|A|$ denotes the number of intervals comprising the set $A$.

To get a suitable upper bound on $|A|$, we utilise the same approach as in lemma 3.1 of \citet{Potgieter}. That is, we consider the probability of an interval being rapid, and use the independence of the events under such consideration to calculate the probability of a certain number of `succesful' events using a binomial distribution. By \citet{Feller2}, if $S_{n}$ denotes the sum of $n$ independent variables which
may take value $1$ with probability $p$ and $0$ with probability
$q=1-p$, then
\begin{equation}\mathbf{P}\{S_{n} \geq r\} \leq
\frac{rq}{(r-n p)^2},\end{equation} when $r>n p$.

Using this, we can estimate the probability of $|A |$ having more than $N^{1-\beta^2-\gamma}$ elements for some $\gamma>0$ (where $1-\beta^2-\gamma >0$ ) by
\begin{eqnarray}
\mathbf{P}\{ |A | \geq N^{1-\beta^2-\gamma} \} &\leq & 
\frac{N^{1-\beta^2-\gamma}(1-N^{-\beta^2(1+o(1))})}{
(N^{1-\beta^2-\gamma}-N^{1-\beta^2 (1+o(1))})^2} \nonumber\\
%&\leq & \frac{N^{-\varepsilon}}{N^{1-\delta^2}(N^{\varepsilon}-N^{-\beta^2 o(1)})^2} %\nonumber\\
&\leq & \frac{1}{N^{1-\beta^2 -\gamma}(1-N^{1-\beta^2(1+o(1))}N^{-1+\beta^2\gamma})^2} \nonumber \\
&\leq & \frac{1}{N^{1-\beta^2 -\gamma}} \label{eq:2.14},
\end{eqnarray} 
where we require that $-\beta^2(1+o(1))+\delta^2+\gamma>0$.
 
We now obtain the estimate
\begin{eqnarray}
\| \mu_1 \|&\leq & N^{-(1-\delta^2 -\gamma)}N^{\delta^2 -1} N^{o(1)}
\nonumber \\ &\leq & N^{o(1)}N^{-\gamma}.
\end{eqnarray}

We assume that $N$ is large enough so that $\| \mu_1 \| \leq 1$. Since the  
construction can be accomplished similarly for $N^2$, $N^3$, \dots, we conclude 
that the measures can be uniformly bounded from above for each stage of the construction. 
However, we still need to ensure that the measures to not decay to $0$, rendering them trivial.
To do so, we use the method of lemma 3.3 from \citet{Potgieter}, which again bounds the number of intervals, this time from below. Again according to \citet{Feller2}, as long as $r<mp$,
\begin{equation}
\mathbf{P}\left\{S_m \leq r\right\}
\leq  \frac{(m-r)p}{(mp-r)^2}.
\end{equation}

We use this to estimate the probability of the number of successes being less than
\begin{equation} \frac{N^{1-\delta^2}}{4\sqrt{2\log N}} \end{equation}
for some constant $\delta>0$ independent of $N$. If this is small enough, the measure will tend to $0$ for only a small proportion of sample paths.

Firstly, we need to ensure the the approximation to the required number of successes
may be used; i.e., we must verify that the following inequality holds:
\begin{equation}\label{eq:2.17}
\frac{N^{1-\delta^2}}{4\sqrt{2\log N}}<Np_1.
\end{equation}  

Using the left-hand side of the approximation \ref{eq:2.5}, we get an upper bound of 
\begin{equation}
p_1 \geq \left( \frac{1}{(2\log N)^{\frac{1}{2}}}-\frac{1}{(2\log N)^{\frac{3}{2}}} \right) N^{-\delta^2}.
\end{equation}

For $N>e$, we then have that $p_1 \geq N^{-\delta^2}(2\sqrt{2\log N})^{-1}$, verifying that 
equation \ref{eq:2.17} holds. Using the stated values of $r$ and $p_1$ (as in \ref{eq:2.17}), we find that the probability of more than $r$ intervals occurring is less than $N^{-1}\sqrt{2\log N}$. We can conclude that  
\begin{eqnarray}
\| \mu_1 \| &\geq & \frac{N^{1-\delta^2}}{4\sqrt{2\log N}} p_{1}^{-1}N^{-1} \nonumber \\ &\geq &  \frac{N^{1-\delta^2}}{4\sqrt{2\log N}} N^{-1+\delta^2} \sqrt{2\log N} = \frac{1}{4},
\end{eqnarray}
with a probability greater than $1-N^{-1}\sqrt{2\log N}$.
 
By repeating this for higher powers of $N$, we see that the lower bound will hold for any further stages of the construction. The probabilities of the exceptional sets are small enough to form a convergent series, the sum of which can be made arbitrarily small by starting with some larger value of $N$ (as expressed more completely in \ref{eq:2.20}). 

The question now becomes whether the measures constructed converge to a measure with the required properties on a subset of the set of $\alpha$-rapid points. If we can show that the measures converge strongly on the ring consisting of all intervals, this will suffice to prove convergence on the Borel sets.  

Since the measures are all defined as step functions on intervals of the form $[(k-1)N^{-m}, kN^{-m}]$, $k=1,2,\dots ,N^m$, $m=1,2,\dots$, we only need to show convergence on such.  
We suppose that an interval $I$ of length $N^{-(n+1)}$ is chosen in step $n+1$ of the construction. This implies that it is a subinterval of some $[(k-1)N^{-n}, kN^{-n}]$ which forms part of the support of $\mu_n$, since the construction guarantees that the measures have nested supports. If we can show that 
\begin{equation}\label{eq:2.18}
\mu_n ([(k-1)N^{-n}, kN^{-n}]) \leq \mu_{n+1}(I),
\end{equation} 
it will show that the measures increase on the intervals that survive each stage. By the construction of the measures,
\ref{eq:2.18} will hold if $p_{n+1}^{-1} N^{-1} \geq p_{n}^{-1}$, that is, if 
\begin{equation} \label{eq:2.19} p_n N^{-1} \leq p_{n+1}.
\end{equation} We once again use the approximation \ref{eq:2.5} to find that 
\begin{eqnarray}\label{eq:2.25}
\sqrt{2\pi }p_{n+1} &\geq & \left( \frac{1}{(2\log N^{\beta^2 (n+1)})^{\frac{1}{2}}}-\frac{1}{(2\log N^{\beta^2 (n+1)})^{\frac{3}{2}}} \right) N^{-\beta^2 (n+1)} \nonumber\\
&\geq & \frac{(n+1)\beta^2 \log N}{(2(n+1)\beta^2 \log N)^{\frac{3}{2}}}N^{-\beta^2 (n+1)}
 = \frac{N^{-\beta^2 (n+1)}}{2^{\frac{3}{2}}\beta ((n+1)\log N)^{\frac{1}{2}}}\\
\sqrt{2\pi }p_n &\leq & \frac{N^{-\beta^2 n}}{(2n \beta^2 \log N^n)^{\frac{1}{2}}}.\label{eq:2.26}
\end{eqnarray}
It is easily verified that \ref{eq:2.19} holds. This implies that the measure of each subinterval that forms part of the construction at the next stage increases, since the measures are uniformly distributed over intervals. Because the total measure as well as the number of intervals are bounded from above at each stage, this implies a convergence of measures on each collection of nested intervals, and hence strong convergence of the measures.  

We must ensure that the set on which the desired convergence does not necessarily take place is suitably small. The first exceptional set to consider is that on which lemma \ref{lem:2.2} does not hold. We can suppose that this set has measure of no more than some $\eta <1/2$. The next stage of the construction then had an exceptional set of measure no greater than $\eta^2$, and so on, with an exceptional probability of $\eta^n$ at stage $n$. The next instance where an exceptional set may occur is in the approximations of the number of intervals at each stage used to show the boundedness of the measures. The probability of there being more than $N^{n(1-\delta^2- \gamma)}$ intervals of the desired form is calculated (as in \ref{eq:2.14}) to be less than $N^{-n(1-\delta^2 -\gamma)}$ -- this is the part which guarantees an upper bound on the measures at each stage. In setting the upper bound, we find an exceptional probability of more than $N^{1-\delta^2}(4\sqrt{2\log N})^{-1}$ successful intervals of less than $N^{-1} \sqrt{2\log N}$. Summarising these results, we see that the exceptional probability for the entire construction is bounded above by the series with the $n$th term equal to
\begin{equation}\label{eq:2.20}
\eta^n + \frac{\sqrt{2\log N^n}}{N^n}+N^{-n(1-\delta^2 -\gamma)}.
\end{equation}

The series is clearly convergent, and the sum can be made smaller than any positive number by starting with a sufficiently large value of $N$. 

In the final step, we must still ascertain that the limit measure has the desired properties. Firstly, it must be established that the Fourier transforms of the measures in the construction converge as well, and that this limit is in fact the Fourier transform of our desired measure.

We know from the continuity theorem (see, for instance, p303 of \citet{Billingsley}) that, for each $u$, $\hat{\mu}_n (u) \to \hat{\mu}(u)$ as $n\to \infty$. By Lemma 2.2 and the argument following it, we know that for a certain given $\varepsilon$,
\begin{equation}
|\hat{\mu}_n (u) - \hat{\mu}_0 (u)|<\varepsilon u^{\frac{1}{2}(\alpha^2 -1)}, \textrm{ for } |u|>1
\end{equation}
for all $n$. Since $|\hat{\mu}_n (u) - \hat{\mu}_0 (u)|$ converges to $|\hat{\mu} (u) - \hat{\mu}_0 (u)|$ for each $u$, it must follow that 
\begin{equation}
|\hat{\mu} (u) - \hat{\mu}_0 (u)|<2\varepsilon u^{\frac{1}{2}(\alpha^2 -1)} \textrm{ for $|u|$ large enough}.
\end{equation}

This implies that
\begin{equation}
|\hat{\mu} (u) - \hat{\mu}_0 (u)| = o(u^{-\frac{\gamma}{2}})
\end{equation}
for all $\gamma < 1-\alpha^2$. 

From this, it follows that, for large enough $u$ and any $\gamma < 1-\alpha^2$, 
\begin{eqnarray}
|\hat{\mu}(u)| &<& |\hat{\mu}_0(u)| +o(u^{-\frac{\gamma}{2}})\nonumber \\ &=& O(u^{-1})+ o(u^{-\frac{\gamma}{2}})  = o(u^{-\frac{\gamma}{2}}).
\end{eqnarray}

For every $\gamma <1-\alpha^2$ then, we have constructed a measure $\mu$ which satisfies
\begin{equation}
|\hat{\mu} (u)| = o(u^{-\frac{\gamma}{2}}),
\end{equation}
with a probability made as close as desired to $1$ by increasing the starting value of $N$. Since starting with higher powers of $N$ decreases the exceptional probability but leads to the same measure in the limit, we can conclude that the desired property holds for the limit measure with probability $1$. This proves Theorem \ref{thm:1.1}.

It should be clear from the above proof that the method is reliant not on any specific properties of Brownian motion other than that the increments under consideration are stationary and independent, and that the probability of any interval under consideration is chosen is bounded below by the length of the interval to a constant power. We exploit this method in the next section, again in the context of Brownian motion, although it can clearly be generalised to other processes. 

\section{Functionally determined rapid points}

Brownian rapid points can also be determined by means of functions rather than constants. Results on various dimensional properties of such sets can be found in \citet{KhosPerX}. Consider the space $C_{ac}[0,1]$ of absolutely continuous functions $f:[0,1]\to \mathbb{R}$ with $f(0)=0$, endowed with the supremum norm. We also consider the space $\mathbb{B}[0,1]$ of bounded functions on $[0,1]$. An element of $C_{ac}[0,1]$ is said to have finite energy is it has finite Sobolev norm:
\begin{equation}
\| f\|_{\mathbb{H}} := \left( \int_{0}^{1} (f'(s))^2 ds \right)^{\frac{1}{2}}<\infty.
\end{equation}
The Hilbert space of functions $f\in  C_{ac}[0,1]$ with finite energy is denoted by $\mathbb{H}$. 
Let $X$ be one-dimensional Brownian motion as usual. \citet{Strassen} proved a version of \ref{eq:1.2} for functions of finite energy. For $f\in \mathbb{H}$ with Sobolev norm at most $1$ and any $t\in [0,1]$,
\begin{equation}
\liminf_{h\to 0^{+}} \sup_{0\leq s \leq 1} \left| \frac{X(t+sh)-X(t)}{\sqrt{2h\log |\log{h^{-1}}|}} - f(s)\right| = 0, \quad {\textrm{ a.s.}}
\end{equation}
We will call a point $t\in [0,1]$ $f$-rapid if
\begin{equation}
\liminf_{h\to 0^{+}} \sup_{0\leq s \leq 1} \left| \frac{X(t+sh)-X(t)}{\sqrt{2h\log{h^{-1}}}} - f(s)\right| = 0
\end{equation}
From this point on, we shall assume that $f\in \mathbb{H}$. We denote the set of $f$-rapid points for a sample path $\omega$ by $S_{\omega}(f)$ and, as usual, omitting the $\omega$ when unnecessary. 
The set $S(f) \subseteq [0,1]$ of $f$-rapid points has Hausdorff dimension $1-\| f\|^{2}_{\mathbb{H}}$ \citep{DeheuvelsMason}.
Using the methods of the previous section, we can now show that the Fourier dimension is the same.

The first step is to describe the set in question, or a large enough subset thereof, as the intersection of a union of equal subintervals of $[0,1]$. As in the proof of Theorem 1.2, at stage $j$ we subdivide the interval $[0,1]$ into $2^{j}$ equal intervals. We designate an interval $[(k-1)2^{-j}, k2^{-j}]$, $k \in \{1,2,\dots ,2^j \}$, as $f, \varepsilon_j$-rapid if   
\begin{equation}\label{eq:2.21}
\exists t_1 <t_2 \in [(k-1)2^{-j}, k2^{-j}] \forall s\in [0,1] \left| \frac{ X(t_1+s(t_2 -t_1))-X(t_1) }{\sqrt{(t_2 -t_1) \log (t_2 -t_1)^{-1}}} - f(s) \right|< \varepsilon_j.
\end{equation}
We denote the random set of $f, \varepsilon_j$-rapid intervals by $\{ I_{i}^{j} (\omega)\}_{i=1}^{2^j}$. For a suitable selection of $\varepsilon_j \to 0$ and accompanying interval lengths $h_j$, as $j\to \infty$, 
\begin{equation}
T_{\omega}(f)= \limsup I_{i}^{j} =  \bigcap_{j=1}^{\infty} \bigcup_{i=1}^{2^j} I_{i}^{j} (\omega) = S_{\omega}(f).
\end{equation} 
The \emph{limsup} is valid in the above because the intervals are once again nested. 

It is clear that each random set $I_{i}^{j}(\omega)$ is independent, and by following the method of the previous section we can find the Fourier dimension if we can find a lower bound on the probability of each interval. To do this, we will use a result of \citet{Schilder}. Before stating it, we require some notation, and we follow the formulation in \citet{DeheuvelsMason}: 
\begin{eqnarray}
\mathbb{S}_r & = & \{f \in \mathbb{B}[0,1]: f\in C_{ac}[0,1],\, \| f\|_{\mathbb{H}}\leq r\}, \quad r>0\\
S_{\Lambda}^{\varepsilon} & = & \{ f\in \mathbb{B}[0,1]: \| f-g \|_{\infty}<\varepsilon \textrm{ for some } g\in \mathbb{S}_{\Lambda} \}.
\end{eqnarray}
We define the function $J$ as
\begin{equation}
J(f) = \left\{ \begin{array}{ll} \int_{0}^{1} (f'(t))^2 dt & \textrm{ if } f\in C_{ac}[0,1] \\
\infty & \textrm{ otherwise}
\end{array} \right. 
\end{equation}
For a set $F\subseteq \mathbb{B}[0,1]$, 
\begin{equation}
J(F) = \inf_{f\in F } J(f) \textrm{ if } F\neq \emptyset, \textrm{ and } J(F)=\infty \textrm{ otherwise.}
\end{equation}

We can now state the theorem which will, when suitably interpreted, yield the probabilities for intervals to be $f,\varepsilon$-rapid.
\begin{thm} \citep{DeheuvelsMason}
For each closed subset $F$ of $\mathbb{B}[0,1]$, we have
\begin{equation}
\limsup_{\lambda \to \infty} \lambda^{-1} \log \mathbf{P}\left( \frac{X(\lambda s)}{2^{\frac{1}{2}} \lambda}
\in F \right) \leq -J(F),
\end{equation}
and for each open $G \subseteq \mathbb{B}[0,1]$
\begin{equation}
\liminf_{\lambda \to \infty} \lambda^{-1} \log \mathbf{P}\left( \frac{X(\lambda s)}{2^{\frac{1}{2}} \lambda}
\in G \right) \geq -J(G)
\end{equation}
\end{thm}
By the scaling property, for $\lambda$, $h >0$, the processes $2^{-1/2}\lambda^{-1}X(\lambda \chi_{0,1})$ and $(2h\lambda)^{-1/2}X(h \chi_{[0,1]})$ are identically distributed, where $\chi_{[0,1]}$ denotes the indicator function of the unit interval.

Using the stationarity of Brownian motion, we have that, for fixed $t\in [0,1]$, $h=2^{-j}$ and $\lambda = \log 2^j$,
\begin{equation}
\mathbf{P}\left[ \forall 0\leq s\leq 1 \left|  \frac{X(t+s2^{-j})-X(t) }{\sqrt{2^{-j+1} \log 2^j}} - f(s) \right| < \varepsilon_j\right] = 
\mathbf{P} \left[ \forall 0\leq s\leq 1  \left|  \frac{ X(s2^{-j}) }{\sqrt{2^{-j+1} \log 2^j}} - f(s) \right| < \varepsilon_j\right].
\end{equation}
Letting $\Lambda = \| f\|_{\mathbb{H}}$, we set $G_j = \mathbb{S}_{\Lambda}^{\varepsilon_j/2}$, which is indeed open in $\mathbb{B}[0,1]$. We clearly have the inclusion
\begin{equation}
\left\{ \omega \in \Omega: \frac{ X(sh) }{\sqrt{2^{-j+1} \log 2^j}}  \in G_j \right\} \subseteq \left\{\omega \in \Omega: \forall 0\leq s \leq 1 \left|  \frac{ X(sh) }{\sqrt{2^{-j+1} \log 2^j}} - f(s) \right| < \varepsilon_j \right\} .
\end{equation}
Since $f\in G_j$, we have that $J(G_j ) \leq \| f\|_{\mathbb{H}}^{2}$. As such, by the second part of the previous theorem,
\begin{equation}
\liminf_{h \to 0^{+}} \lambda^{-1}\log \mathbf{P} \left[ \forall \leq s\leq 1  \left|  \frac{ X(sh) }{\sqrt{2^{-j+1} \log 2^j}} - f(s) \right| < \varepsilon_j\right] \geq -\| f\|_{\mathbb{H}}^2.
\end{equation}
This implies that we have a sequence $\{ h_j \}_{j=1}^{\infty}$ such that, for each $j$,
\begin{equation}
\mathbf{P} \left[ \forall 0\leq s\leq 1  \left|  \frac{ X(sh_j) }{\sqrt{2^{-j+1} \log 2^j}} - f(s) \right|< \varepsilon_j\right] \geq e^{-\| f\|_{\mathbb{H}}^{2}\log h_{j}^{-1}}= h_{j}^{\| f\|_{\mathbb{H}}^{2}.}
\end{equation}
With this inequality in place, we can apply the methods of the previous sections. We have therefore proved the following theorem:
\begin{thm}
$S(f)$, the $f$-rapid points of a one-dimensional Brownian motion $X:\Omega \times [0,1] \to \mathbb{R}$, have Fourier dimension $1-\| f\|^{2}_{\mathbb{H}}$, almost surely, and is therefore a Salem set.
\end{thm}
In conclusion, we should point out that this method for determining the Fourier dimension of random sets can be applied to any similar construction, independent of whether these sets are determined by Brownian motion or not.

\end{document}